\documentclass[11pt]{amsproc}
\usepackage{mathrsfs, color}
\usepackage{stmaryrd}
\usepackage{cases}
\usepackage{amssymb}
\usepackage{amsmath}
\usepackage{amsfonts}
\usepackage{graphicx}
\usepackage{amsmath,amstext,amsbsy,amssymb, color}

\usepackage{latexsym}
\usepackage{amsmath}
\usepackage{amssymb}
\usepackage{bm}
\usepackage{graphicx}
\usepackage{wrapfig}
\usepackage{fancybox}

\newtheorem{theorem}{Theorem}[section]
\newtheorem{lemma}[theorem]{Lemma}

\newtheorem{corollary}[theorem]{Corollary}
\theoremstyle{definition}

\theoremstyle{remark}
\newtheorem{remark}[theorem]{Remark}

\numberwithin{equation}{section} \errorcontextlines=0

\newcommand{\diag}{\mbox{diag}}

\newcommand{\Pf}{\mathrm{Pf}}

\newcommand{\per}{\mathrm{per}}

\newcommand{\ot}{\otimes}

\newcommand{\si}{\sigma}

\newcommand{\Mat}{\mathrm{Mat}}

\begin{document}
	
	\title[General Capelli-type identities]
	{General Capelli-type identities}
	
	\author{Naihuan Jing}
	\address{Department of Mathematics, North Carolina State University, Raleigh, NC 27695, USA}
	\email{jing@ncsu.edu}
	
    \author{Yinlong Liu}
	\address{Department of Mathematics, Shanghai University, Shanghai 200444, China}
	\email{yinlongliu@shu.edu.cn}
	
	\author{Jian Zhang}
	\address{School of Mathematics and Statistics,
		Central China Normal University, Wuhan, Hubei 430079, China}
	\email{jzhang@ccnu.edu.cn}
	
	\thanks{{\scriptsize
			\hskip -0.6 true cm MSC (2020): Primary: 17B37 Secondary: 20G05, 17B35, 17B66, 05E10
	}}
	

\begin{abstract} The classical Capelli identity is an important determinantal identity of a matrix with noncommutative entries that
determines the center of the enveloping algebra of the general linear Lie algebra, and was used by Weyl as a main tool to study irreducible representations in his famous book on
classical groups. 

In 1996 Okounkov found higher Capelli identities involving immanants of the generating matrix of $U(\mathfrak{gl}(n))$ which
correspond to arbitrary orthogonal idempotent of the symmetric group. It turns out that Williamson also discovered
a general Capelli identity of immanants for $U(\mathfrak{gl}(n))$ in 1981. In this paper, we use a new method to derive a family of even more general Capelli identities that
include the aforementioned Capelli identities as special cases as well as many other Capelli-type identities as corollaries. In particular, we obtain 
generalized Turnbull's identities for both symmetric and antisymmetric matrices, as well as the generalized Howe-Umeda-Kostant-Sahi identities for antisymmetric matrices which confirm the conjecture of Caracciolo, Sokal, and Sportiello \cite{CSS}.

\end{abstract}
	\maketitle

\section{Introduction}

    The classical Capelli identity \cite{C}
    is one of the most important results in the classical invariant theory that can provide a set of generators for the center of enveloping
    algebra $U(\mathfrak{gl}(n))$.
    Let $E_{ij}$, $1\leq i,j\leq n$, be the basis of $\mathfrak{gl}(n)$ and $E=(E_{ij})_{n\times n}$, then the column determinant
\begin{equation*}
\begin{aligned}
&\det(E+diag(n-1,\ldots,0))
\\
&=\sum_{\sigma\in S_n}sgn(\sigma)
(E_{\sigma(1)1}+(n-1)\delta_{\sigma(1)1})(E_{\sigma(2)2}+(n-2)\delta_{\sigma(2)2})\cdots E_{\sigma(n)n}
\end{aligned}
\end{equation*}
is a central element of $U(\mathfrak{gl}(n))$.
Let $\mathcal{PD}(\mathbb{C}^{n\times n})$ be the algebra generated by $x_{ij}$ and $\partial_{ij}=\partial/\partial{x_{ij}}$, $1\leq i,j\leq n$.
Denote $X=(x_{ij})$, $D=(\partial_{ij})$, then
\begin{equation*}
E \mapsto XD^t
\end{equation*}
defines an algebra homomorphism from $U(\mathfrak{gl}(n))$ to $\mathcal{PD}(\mathbb{C}^{n\times n})$.
The image of the central element was determined by Capelli \cite{C} as follows.
\begin{equation*}
\det(XD^t+diag(n-1,\ldots,0))=\det X\det D^t.
\end{equation*}

The Capelli identity has been extended in various directions.
   Turnbull \cite{Turnbull} proved a Capelli-type identity for determinants of symmetric matrices and permanents of antisymmetric matrices.  Howe and Umeda \cite{HU} and Kostant and Sahi \cite{KS} independently discovered 
 a Capelli-type identity for determinants of antisymmetric matrices. Foata and Zeilberger \cite{FZ} gave short combinatorial proofs of Capelli's and Turnbull's identities, however, their method did not work for the Howe-Umeda-Kostant-Sahi identity.
Molev and Nazarov \cite{Molv1,Molv2,Nazarov} found the Capelli identities for the classical Lie algebra $\mathfrak{so}(n)$ and $\mathfrak{sp}(n)$ as well as the super analogs. The identity has also been generalized for different reductive dual pairs \cite{It1,It2}. Furthermore, the Capelli identities were extended to Manin matrices in \cite{CSS,CFR} and quantum groups in \cite{JZ2, JZ3, NUW1, NUW2}.

  In 1996, Okounkov \cite{Ok, Oo} gave higher Capelli identities 
  for determinants and permanents.
  These generalizations were also proved with different methods by Nazarov \cite{N2} and Molev \cite{Molv3}.
   It turns out that Williamson \cite{Williamson} had also derived general Capelli identities for
   immanants in 1981
   and Williamson's results naturally contain many other Capelli identities as special cases.
   Besides these Capelli identities,
   there is also an open question on how to generalize the Howe-Umeda-Kostant-Sahi identity to give a Capelli-type identity for
   determinants of antisymmetric matrices.

In this paper, we derive generalized Capelli-type identities for immanants that generalize Williamson and Okounkov's Capelli identities, which
   in turn include many previously known identities as special cases, for instances \cite{ CSS,CFR,FZ,Ok, Oo, Turnbull,UT, Wall,Williamson} etc. 
   Our new identities also confirm the aforementioned conjecture of Caracciolo-Sokal-Sportiello
   on Capelli-type identities for determinants of antisymmetric matrices \cite{CSS}.

   The general Capelli identity can be viewed as some deformation of Williamson's identity via a parametric matrix $H$,
   in which Williamson's identity is the special
   case of $H=I$ the identity matrix. 

   Our method is a deformation of Molev's method \cite{Molv3} in addition to ideas from the Yangian algebra.
   The proof is based on the properties of Jucys-Murphy elements in group algebra of the symmetric group.  We  do not need to use
the Wick formula \cite{Ok, Oo}
or Laplace expansion for Schur matrix functions which was one of the techniques in Williamson's method \cite{Williamson}.
The general Capelli identity can be viewed as some deformation of Okounkov and Williamson's identities via a parametric matrix $H$,
   where Williamson's identity is the special
   case of $H=I$ the identity matrix.
In this way, we naturally obtain a family of general Capelli-type identities, which include both Williamson's Capelli identities
   and many other Capelli identities as special cases.

   The layout of the paper is as follows.
 In Section \ref{s:gen-capelli-identities} we study the Capelli identities for immanants that generalize Okounkov's and Williamson's results  \cite{Ok, Oo, Williamson}. In Section \ref{s:gen-Turnbull identities} we generalize the Turnbull identities for the determinant of symmetric matrices and permanent of antisymmetric matrices. In Section \ref{s:gen-HUKS-identity} we obtain generalizations of the Howe-Umeda-Kostant-Sahi  identity which were  conjectured
by Caracciolo, Sokal and Sportiello \cite{CSS}.

\section{General Capelli identity}\label{s:gen-capelli-identities}
The higher Capelli identity was obtained by Okounkov \cite{Ok}, then he gave generalization of higher Capelli identities in \cite{Oo}.
In 1981 Williamson \cite{Williamson} obtained a general Capelli-type identity for immanants.

In this section, we will derive general Capelli identities which generalize both Williamson's and Okounkov's Capelli identities.

Let $\mathfrak{S}_r$ be the symmetric group, and let
$z_{k}$ be the usual Jucys-Murphy elements
\begin{equation*}
z_1=0, \qquad	z_k=\sum_{i=1}^{k-1}(i,k),\ 2\leq k \leq r,
\end{equation*}
where $z_k$ commutes with the symmetric group $\mathbb C\mathfrak{S}_{k-1}$ on $\{1, \ldots, k-1\}$.
We also need the reverse Jucys-Murphy elements $  {z}_{k}'$ given by
\begin{equation}
	 {z}_{k}'=\sum_{i=k+1}^{r}(k,i),\ 1\leq k \leq r-1,\qquad   z_r'=0,
\end{equation}
where $ {z}_{k}'$ commutes with the symmetric group $\mathbb C\mathfrak{S}_{r-k}$ on $\{k+1, \ldots,r\}$.


Let $\sigma\mapsto P^{\sigma}$ be the natural representation of $\mathfrak{S}_r$ on the
space $(\mathbb C^m)^{\otimes r}$ by permuting
the factors.
Define the following element in the algebra $\mathbb{C}\mathfrak{S}_r \ot \text{End}((\mathbb{C}^{m})^{\ot r})$:
\begin{equation}\label{e:idem}
     \mathcal{E}^{(r)}=\sum_{\si \in \mathfrak{S}_r} \si \ot P^{\si}.  
\end{equation}
More generally, we also extend the meaning of $P^{\sigma}$ by considering the permutation operator $P_{s\times m}\in \mathrm{Hom}(\mathbb{C}^{s}\otimes \mathbb{C}^{m}, \mathbb{C}^{m}\otimes \mathbb{C}^{s})$
defined by $P_{s\times m}(u\otimes v)=v\otimes u$, $u\in\mathbb C^s, v\in\mathbb C^m$. Explicitly
\begin{equation}\label{e:P}
P_{s \times m}=\sum_{i=1}^m \sum_{j=1}^s e_{ij}\otimes e_{ji},
\end{equation}
where $e_{ij}$ (resp. $e_{ji}$) are the standard matrix units in $\mathrm{Hom}(\mathbb C^s, \mathbb C^m)$ (resp. $\mathrm{Hom}$ $(\mathbb C^m, \mathbb C^s))$.

For convenience, we will denote $P_{s \times m} $ by $P $ if there is no confusion in the context.

Let $\mathcal A$ be an  algebra, we denote by $\Mat_{n \times m}(\mathcal A)$ the set of $n\times m$ matrices whose entries are in $\mathcal A$.
For $n=m$, we simply write $\Mat_{n \times m}(\mathcal A)$ by $\Mat_{n}(\mathcal A)$.
Throughout of the paper, we consider mostly matrices with entries in $\mathcal A$.
For any matrix $X=(X_{ij})$ in $\Mat_{n \times m}(\mathcal A)$, one can regard $X$ as an element of
$\mathcal A\otimes \mathrm{Hom}(\mathbb C^m, \mathbb C^n)$. That is,
\begin{equation}
  X=\sum_{1\leq i\leq n,\atop 1\leq j\leq m}X_{ij}\otimes e_{ij},
\end{equation}
where $e_{ij}$ are the standard matrix units in $\mathrm{Hom}(\mathbb C^m, \mathbb C^n)$.
For any $1\leq r\leq s$, we define $X_r$ the element in
$\mathcal A\otimes \mathrm{Hom}\big((\mathbb C^m)^{\otimes s},(\mathbb C^m)^{\otimes r-1}\otimes  \mathbb C^n\otimes \mathbb (C^m)^{\otimes s-r})\big)$
as follows,
\begin{equation}
   X_r=\sum_{1\leq i\leq n,\atop 1\leq j\leq m}X_{ij}\otimes 1^{\otimes r-1}\otimes e_{ij}\otimes 1^{\otimes s-r},
\end{equation}
where $1$ is the identity matrix in $\mathrm{End}(\mathbb {C}^m)$.


The following theorem is a generalization of the Capelli identity for immanants, which contains both Williamson's generalized
Capelli identity \cite{Williamson} and Okounkov's higher Capelli identities \cite{Oo} as special cases.

\begin{theorem}\label{Capelli immanant}
	Let $\mathcal A$ be an algebra, and let $X=(X_{ij}) \in \Mat_{n\times m}(\mathcal A)$, $Y=(Y_{ij})\in \Mat_{m \times s}(\mathcal A)$, where the entries of $X$ are mutually commutative and  those of $Y$ are arbitrary.  Suppose in
	$\mathcal A\otimes  \mathrm{Hom}(\mathbb C^{s}\otimes \mathbb C^{m}, \mathbb C^{m}\otimes \mathbb C^{n})$
\begin{align}\label{cape-rela1}
		&Y_1X_2-X_2Y_1=H_2P,
\end{align}
where $H$ is an $n\times s$ matrix. 
In terms of entries the relation \eqref{cape-rela1} is 
	\begin{equation}
		[X_{ij},Y_{kl}]=-\delta_{jk}H_{il}, \quad \quad 1\leq i \leq n,1\leq j,k \leq m , 1\leq l \leq s.
	\end{equation}
For $r>m=1$, we assume that
\begin{align}\label{cape-rela11}
		&X_1H_2=H_2X_1.
\end{align}
	Then the following identity holds in $\mathbb C\mathfrak{S}_r\ot \mathcal A\otimes \mathrm{Hom}( (\mathbb C^s)^{\ot r} ,  (\mathbb C^n)^{\otimes r})$ for any $r, m, n$.
	\begin{equation}\label{cape}
\begin{split}
&(X_1Y_1)(X_2Y_2-z_2\ot H_2) \cdots (X_rY_r-z_r\ot H_r)\mathcal{E}^{(r)}\\
&=\mathcal{E}^{(r)}(X_1Y_1-{z}_1'\ot H_1)\cdots (X_{r-1}Y_{r-1}-{z}_{r-1}'\ot H_{r-1})(X_rY_r)
 \\
& = X_1\cdots X_r Y_1\cdots Y_r \mathcal{E}^{(r)}\\
& =\mathcal{E}^{(r)} X_1\cdots X_r Y_1\cdots Y_r.
\end{split}		
	\end{equation}

\end{theorem}
\begin{proof}
If $m\geq 2$, the relation \eqref{cape-rela1} implies \eqref{cape-rela11}.
 Indeed, by the Jacobi identity for any indices $i,k\in [n]$, $j\neq l\in [m]$ and $t\in [s]$,
	\begin{equation}
		\begin{aligned}
			&[X_{ij},H_{kt}]=-[X_{ij},[X_{kl},Y_{lt}]]\\
			&=-[[X_{ij},X_{kl}],Y_{lt}]-[X_{kl},[X_{ij},Y_{lt}]]\\
			&=0.
		\end{aligned}
	\end{equation}

We will prove the following equation by induction on $r$.
\begin{equation}
\begin{split}
&\mathcal{E}^{(r)}(X_1Y_1-{z}_1'\ot H_1)\cdots (X_{r-1}Y_{r-1}-{z}_{r-1}'\ot H_{r-1})(X_rY_r)
 \\
& =\mathcal{E}^{(r)} X_1\cdots X_r Y_1\cdots Y_r.\\
\end{split}		
	\end{equation}

It is obvious for $r=1$. If $r >1$, we note that
$(r-1)! \mathcal{E}^{(r)}=\mathcal{E}^{(r)}\mathcal{E}^{(r-1)}$,
where \begin{equation}
     \mathcal{E}^{(r-1)}=\sum_{\si \in \mathfrak{S}_{r-1}} \si \ot P^{\si},
\end{equation}
and  $\mathfrak{S}_{r-1}$ is the symmetric group  on $\{2,3,\ldots,r\}$.
Then $\mathcal{E}^{(r-1)}$ commutes with the Jucy-Murphy element $z_1'\otimes 1$.
It follows from the induction hypothesis that
	\begin{equation}
\begin{split}
&\mathcal{E}^{(r)}(X_1Y_1-{z}_1'\ot H_1)\cdots (X_{r-1}Y_{r-1}-{z}_{r-1}'\ot H_{r-1})(X_rY_r)
 \\
&= \frac{1}{(r-1)!}\mathcal{E}^{(r)}(X_1Y_1-{z}_1'\ot H_1)\mathcal{E}^{(r-1)}(X_2Y_2-{z}_2'\ot H_2)\cdots (X_rY_r)
\\
&= \frac{1}{(r-1)!}\mathcal{E}^{(r)}(X_1Y_1-{z}_1'\ot H_1)\mathcal{E}^{(r-1)}X_2\cdots X_r Y_2\cdots Y_r
\\
&= \mathcal{E}^{(r)}(X_1Y_1-{z}_1'\ot H_1) X_2\cdots X_r Y_2\cdots Y_r.
\end{split}		
	\end{equation}

The following identity can be verified by direct computation.
 \begin{equation}
\mathcal{E}^{(r)}\left((i,j)\otimes 1\right)=\mathcal{E}^{(r)}\left(1\otimes P^{(i,j)}\right).
 \end{equation}

Therefore,
	\begin{equation}\label{cape2}
		\begin{aligned}
&  \mathcal{E}^{(r)}(X_1Y_1-{z}_1'\ot H_1)\cdots (X_{r-1}Y_{r-1}-{z}_{r-1}'\ot H_{r-1})(X_rY_r)   \\
&=\mathcal{E}^{(r)}(X_1Y_1-\sum_{k=2}^{r}P^{(1,k)} H_1)X_2\cdots X_r Y_2\cdots Y_r\\
			&=\mathcal{E}^{(r)}X_1\cdots X_r Y_1\cdots Y_r\\
			&\quad +\sum_{i=2}^{r}\mathcal{E}^{(r)} X_1X_2\cdots X_{i-1}H_i P^{(1,i)} X_{i+1} \cdots X_r Y_2\cdots Y_r\\
			&\quad - \sum_{k=2}^{r}\mathcal{E}^{(r)} P^{(1,k)} H_1X_2\cdots X_rY_2\cdots Y_r.
		\end{aligned}
	\end{equation}
Using the relation in $\mathrm{End}( (\mathbb C^s),  (\mathbb C^m)^{\otimes i})$
\begin{equation}
	\begin{aligned}
	   P^{(1,2)}P^{(2,3)}\cdots P^{(i-1,i)} P^{(1,i)}= P^{(2,3)} P^{(3,4)}\cdots P^{(i-1,i)},
	\end{aligned}
\end{equation}
we have that
	\begin{equation}
		\begin{aligned}
			&X_1X_2\cdots X_{i-1}H_iP^{(1,i)}\\
			&= P^{(i-1,i)}\cdots P^{(1,2)}H_1X_2\cdots X_{i-1}X_i P^{(1,2)}\cdots P^{(i-1,i)}P^{(1,i)}\\
&=  P^{(i-1,i)}\cdots P^{(1,2)} H_1 X_2 \cdots X_{i-1} X_i P^{(2,3)}P^{(3,4)}\cdots P^{(i-1,i)}\\
&= P^{(1,i)}H_1X_2\cdots X_{i-1}X_i.
		\end{aligned}
	\end{equation}
for any $2 \leq  i\leq r$.

The equation
\begin{equation}\label{okou-cape}
\begin{split}
&(X_1Y_1)(X_2Y_2-z_2\ot H_2) \cdots (X_rY_r-z_r\ot H_r)\mathcal{E}^{(r)}\\
& = X_1\cdots X_r Y_1\cdots Y_r \mathcal{E}^{(r)}\\
\end{split}		
	\end{equation}
can be proved similarly.
And it is easy to see that
\begin{align*}
&X_1\cdots X_r Y_1\cdots Y_r \mathcal{E}^{(r)} =\mathcal{E}^{(r)} X_1\cdots X_r Y_1\cdots Y_r.
\end{align*}
Thus we obtain the theorem.
\end{proof}

\begin{remark}
Let  $\varphi: \mathfrak{S}_r\rightarrow \mathrm{GL}(V)$ be a representation of the symmetric group $\mathfrak{S}_r$.
We denote
	\[
	\mathcal{E}_{\varphi}^{(r)}=\sum_{\si \in \mathfrak{S}_r} \varphi(\si) \ot P^{\si}.
	\]
Then for multi-index $I=(i_1, \ldots, i_r)$ on $[m]=\{1, \ldots, m\}$, the coefficient of $e_{i_1}\otimes\cdots \otimes e_{i_r}$ in
	\begin{equation}
		\mathcal{E}_{\varphi}^{(r)} X_1\cdots X_r e_{k_1}\otimes\cdots \otimes e_{k_r}
	\end{equation}
	is the Schur matrix function:
	\[
	d^{\varphi}(X_{IK})= \sum_{\si \in \mathfrak{S}_r} \varphi(\si) \prod_{t=1}^{r} X_{i_{\si(t)},k_t},
	\]
whose trace is the immanant.
	
	Now define the map $\alpha_{J}: [r]\rightarrow [m]$ by $ \alpha_{J}(k)=j_k$ for $1 \leq k \leq r$. Let $v(\alpha_{J})=|\alpha^{-1}(1) |! |\alpha^{-1}(2)| ! \cdots | \alpha^{-1}(m) | ! $
for any $r, m, n$.

Apply $ \varphi$ to both sides of  the equation
	\begin{equation}\label{cape}
\begin{split}
&\mathcal{E}^{(r)}(X_1Y_1-{z}_1'\ot H_1)\cdots (X_{r-1}Y_{r-1}-{z}_{r-1}'\ot H_{r-1})(X_rY_r)
\\
& =\mathcal{E}^{(r)} X_1\cdots X_r Y_1\cdots Y_r.
\end{split}		
	\end{equation}
and then act
on $e_{k_1}\otimes \cdots \otimes e_{k_r}$, by comparing the coefficients of
$e_{i_1}\otimes \cdots \otimes e_{i_r}$ we
 obtain the following equation
	\begin{equation}\label{capelli williamson}
		\begin{aligned}
			&d^{\varphi}((XY)_{IK}-diag(\varphi(z_1'),\ldots,\varphi(z_r'))H_{IK})
			\\
			&=\sum\limits_{J} \dfrac{1}{v(\alpha_{J})} d^{\varphi}({X_{IJ}})d^{\varphi}({Y_{JK}}),
		\end{aligned}
	\end{equation}
	where $I=(i_1,\ldots ,i_r)$ and $K=(k_1,\ldots,k_r)$ are two multi-indices and the sum is over all indices of non-decreasing integers $J=(j_1,\ldots,j_r)$.
	In particular, if $s=n$ and taking $H=I$ in relation \eqref{capelli williamson}, we obtain Williamson's general Capelli identity in \cite{Williamson}.
\end{remark}

\begin{remark}

Let $T$ and $T'$ be two Young tableaux of shape $\mu$ and let $v_{T}$ and $v_{T'}$ be the corresponding vector in the Young basis for the irreducible $\mathbb C\mathfrak{S}_r$-module $V^{\mu}$. The action of the Jucys-Murphy elements $ {z}_{k}$ on the vector is described by the content $c_{T}$ of  the Young tableau \cite{OV}:
\begin{equation*}
    {z}_k\cdot v_T=c_{T}(k)v_T.
\end{equation*}

Moreover,
\begin{equation}
	\begin{aligned}
		\langle \mathcal{E}^{(r)}v_{T'},v_{T} \rangle&= \sum_{\si \in \mathfrak{S}_r} \langle\si v_{T'},v_T \rangle \ot P^{\si}\\
     &=\sum_{\si \in \mathfrak{S}_r} \langle v_{T'},\si^{-1}v_{T} \rangle \ot P^{\si}\\
     &= \Psi_{TT'}.
	\end{aligned}
\end{equation}

Applying  both sides of \eqref{okou-cape} to $v_{T'}$  and taking inner product with $v_T $, we obtain that

\begin{equation}\label{okou-cape 1}
	\begin{split}
		&(X_1Y_1)(X_2Y_2-c_{T}(2)\ot H_2)\cdots(X_rY_r-c_{T}(r)\ot H_r)\Psi_{TT'}\\
		& = X_1\cdots X_r Y_1\cdots Y_r\Psi_{TT'}.
	\end{split}		
\end{equation}
If $H=I$, equation \eqref{okou-cape 1} specialize to  the generalized higher Capelli  identities given by Okounkov \cite{Oo}
\begin{equation}\label{okou-cape 2}
	\begin{split}
		&(X_1Y_1)(X_2Y_2-c_{T}(2) )\cdots(X_rY_r-c_{T}(r) )\Psi_{TT'}\\
		& = X_1\cdots X_r Y_1\cdots Y_r\Psi_{TT'}.
	\end{split}		
\end{equation}

\end{remark}

\section{Generalized Turnbull's identities}\label{s:gen-Turnbull identities}
In this section we study the Turnbull-type identities for determinants of symmetric matrices and permanents of antisymmetric matrices.

Let $A_r$ and $S_r$ be the  normalized antisymmetrizer and symmetrizer  on $\text{End}((\mathbb{C}^{m})^{\ot r})$ respectively:
\begin{equation}
	A_r=\frac{1}{r!}\sum_{\sigma\in \mathfrak{S}_r}\text{sgn}\ \sigma P^{\sigma}, \qquad
	S_r=\frac{1}{r!}\sum_{\sigma\in \mathfrak{S}_r} P^{\sigma}.
\end{equation}

They are two special idempotents, 
respectively corresponding to the irreducible  representations of $\mathfrak{S}_r$ associated with
the Young tableaux of shapes $(1^r)$ and $(r)$. Let $X$ be any matrix in $\Mat_{n\times m}(\mathcal A)$. In the former case,
the coefficient of $e_{i_1}\otimes\cdots \otimes e_{i_r}$ in
\begin{equation}
	A_rX_1\cdots X_r e_{k_1}\otimes\cdots \otimes e_{k_r}
\end{equation}
is $\dfrac{1}{r!}\det(X_{IK})$. In the latter case, 
the coefficient of  $e_{i_1}\otimes\cdots \otimes e_{i_r}$  in
\begin{equation}
	S_rX_1\cdots X_r e_{k_1}\otimes\cdots \otimes e_{k_r}
\end{equation}
is
\[
\dfrac{1}{r!}\per(X_{IK})=\dfrac{1}{r!}\sum_{\sigma\in \mathfrak{S}_r}X_{i_{\sigma(1)},k_1}\cdots X_{i_{\sigma(r)},k_r}.
\]

Let $Q=P^{t_1}$ be the partial transpose of the transposition $P$, i.e.
\begin{equation}\label{e:Q}
 Q=\sum_{i=1}^n \sum_{j=1}^n E_{ij}\otimes E_{ij} \in \mathrm{End}(\mathbb{C}^{n}\otimes \mathbb{C}^{n}).
\end{equation}

\begin{lemma}\label{AXQ-SXQ}
Let $X$ be a matrix $\Mat_{n}(\mathcal A)$.
   For $2\leq i \leq r$, we have that in $  \mathcal A \ot \mathrm{End}( (\mathbb C^n)^{\otimes r})$
   \begin{equation}
   	\begin{aligned}
   		&A_rX_1Q^{(1,i)}=0,\ if \mbox{$X_{n\times n}$ is symmetric},\\
   		&S_rX_1Q^{(1,i)}=0,\ if \mbox{$X_{n\times n}$ is anti-symmetric}.
   	\end{aligned}
   \end{equation}
\end{lemma}
\begin{proof} We prove the first identity by applying
$A_rX_1Q^{(1,i)}$ to the vector $e_{j_1,\dots,j_r}=e_{j_1}\ot \cdots \ot e_{j_r}$. 
If $j_1\neq j_i$, $Q^{(1,i)}e_{j_1,\dots,j_r}=0$.
For $j_1= j_i$, we have that
	\begin{equation}
		\begin{aligned}
			&A_rX_1Q^{(1,i)}e_{j_1,\dots,j_r}\\
			&=A_rX_1\sum_{l=1}^{n}e_{l,j_2,\dots,j_{i-1},l,j_{i+1},\dots, j_r}\\
			&=A_r\sum_{k,l=1}^{n} X_{kl}e_{k,j_2,\dots,j_{i-1},l,j_{i+1},\dots, j_r}\\
			&=A_r\dfrac{1-P^{(1,i)}}{2}\sum_{k,l=1}^{n}X_{kl}e_{k,j_2,\dots,j_{i-1},l,j_{i+1},\dots, j_r}\\
			&=A_r\sum_{k,l=1}^{n}\dfrac{X_{kl}-X_{lk}}{2}e_{k,j_2,\dots,j_{i-1},l,j_{i+1},\dots, j_r},
		\end{aligned}
	\end{equation}
 which vanishes for the symmetric matrix $X$. 
 The second identity can be  proved by the same argument.
\end{proof}

The following generalizes  Turnbull's formula \cite{Turnbull, FZ} for symmetric matrices.
\begin{theorem}\label{sym-turnbull}
Let $X=(X_{ij}) \in \Mat_{n}(\mathcal A)$ and $Y=(Y_{ij})\in \Mat_{n \times m}(\mathcal A)$. Suppose $X$ is symmetric with commutative entries and $Y$ is arbitrary.
 Suppose in $\mathcal A\otimes  \mathrm{Hom}(\mathbb C^{m}\otimes \mathbb C^{n}, \mathbb C^{n}\otimes \mathbb C^{n})$
	\begin{align}\label{turnbull-rela1}
	& Y_1X_2-X_2Y_1=(Q+P)H_1,
	\end{align}
where $H$ is an $n\times m$ matrix.
This relation can be written entry-wise as
	\begin{equation}
		[X_{ij},Y_{kl}]=-(\delta_{ik}H_{jl}+\delta_{jk}H_{il}), \quad \quad 1\leq i,j,k \leq n, 1\leq l \leq m.
	\end{equation}
Moreover, for $n=2,3$, we assume that
	\begin{align}\label{turnbull-rela11}
&  X_1H_2=H_2X_1.
	\end{align}
	Then
in $ \mathcal A \ot \mathrm{Hom}( (\mathbb C^m)^{\otimes r},  (\mathbb C^n)^{\otimes r})$ one has
that
	\begin{equation}\label{Turnbull}
		A_r(X_1Y_1+ (r-1)H_1)\cdots (X_rY_r) =A_r X_1\cdots X_r Y_1\cdots Y_r,
	\end{equation}
for any $r, n, m$.
		In particular,
	\begin{equation}\label{capelli symmetric}
		{\det}\big((XY)_{IL}+H_{IL} \diag(r-1,r-2,\ldots,1,0) \big) \\
		=\sum\limits_{J} {\det}(X_{IJ}){\det}(Y_{JL}),
	\end{equation}
	where $I=\{i_1,\ldots ,i_r\} \subset[n]$ and $L=\{l_1,\ldots,l_r\}  \subset[m]$ are two sets of cardinality $r$, and the sum is taken over all subset $J=\{j_1,\ldots,j_r\}$ of $ [n]$.
\end{theorem}
\begin{proof}
	It is sufficient to prove \eqref{Turnbull},
	equation \eqref{capelli symmetric} follows from \eqref{Turnbull}.
For $n\geq 4 $, relation \eqref{turnbull-rela1} implies $X_1H_2=H_2X_1$. Indeed, for any indices $i,j,k,l\in [n]$ with  $l\notin\{ i,j,k\}$   and $t\in [m]$, we have that
	\begin{equation}
		\begin{aligned}
			&[X_{ij},H_{kt}]=-[X_{ij},[X_{kl},Y_{lt}]]\\
			&=-[[X_{ij},X_{kl}],Y_{lt}]-[X_{kl},[X_{ij},Y_{lt}]]\\
						&=0.
		\end{aligned}
	\end{equation}

	We now prove \eqref{Turnbull} by induction on $r$.
	It is trivial for $r=1$.
Using the relation $A_r=A_rA_{r-1}$ and the induction hypothesis, where $A_{r-1}$ is the antisymetrizer on the indices $\{2,3,\ldots,r\}$, we have that
	\begin{equation}
		\begin{aligned}
			&A_r(XY+(r-1)H)_1\cdots (XY)_r \\
			&=A_r(XY+(r-1)H)_1 A_{r-1}(XY+(r-2)H)_2\cdots (XY)_r\\
			&=A_r(XY+(r-1)H)_1A_{r-1}X_2\cdots X_r Y_2\cdots Y_r\\
&=A_r(XY+(r-1)H)_1 X_2\cdots X_r Y_2\cdots Y_r.\\
		\end{aligned}
	\end{equation}
	
Applying the relations \eqref{turnbull-rela1}, we have that
\begin{align*}
				&A_r(XY+(r-1)H)_1X_2\cdots X_r Y_2\cdots Y_r\\
		&=A_rX_1\cdots X_r Y_1\cdots Y_r\\
		&\quad +A_r\sum_{i=2}^{r}X_1X_2\cdots X_{i-1} (Q^{(1,i)}+P^{(1,i)} )H_1X_{i+1}\cdots X_rY_2\cdots Y_r\\
		&\quad + (r-1) A_rH_1X_2\cdots X_rY_2\cdots Y_r.
\end{align*}
Similar to the proof of Theorem \ref{Capelli immanant}, for $2\leq i \leq r$ we have

\begin{equation}\label{eq symm}
	\begin{aligned}
            &A_rX_1X_2\cdots X_{i-1}P^{(1,i)}H_1\\
			&=A_rP^{(1,i)}H_1X_2\cdots X_{i-1}X_i\\
            &=-A_rH_1X_2\cdots X_{i-1}X_i.\\
	\end{aligned}
\end{equation}

Combining \eqref{eq symm} with Lemma \ref{AXQ-SXQ}, we have that
	\begin{equation}
		\begin{aligned}
			&A_r\sum_{i=2}^{r}X_1X_2\cdots X_{i-1} (Q^{(1,i)}+P^{(1,i)})H_1 X_{i+1}\cdots X_rY_2\cdots Y_r\\
			&\quad + (r-1)A_rH_1X_2\cdots X_rY_2\cdots Y_r\\
&=0.
\end{aligned}
	\end{equation}
This completes the proof.
\end{proof}

\begin{remark}
If $H=I$, Theorem \ref{sym-turnbull} is the Turnbull  identity \cite{Turnbull} for symmetric matrices.
For $H=hI$, Theorem \ref{sym-turnbull}  specializes to Turnbull's identity for symmetric matrices given in \cite{FZ} and \cite{CSS,CFR}.
\end{remark}

The following theorem is a generalization of Turnbull's identity for the permanents of
 anti-symmetric matrices.
\begin{theorem} \label{antisym-turnbull}
Let $X=(X_{ij}) \in \Mat_{n}(\mathcal A)$, $Y=(Y_{ij})\in \Mat_{n \times m}(\mathcal A)$. Suppose $X$ is antisymmetric with commutative entries and
$Y$ is arbitrary.  Assume  in  $\mathcal A\otimes  \mathrm{Hom}(\mathbb C^{m}\otimes \mathbb C^{n}, \mathbb C^{n}\otimes \mathbb C^{n})$
	\begin{equation}\label{antisym-rela1}
		Y_1X_2-X_2Y_1=(Q-P)H_1,
	\end{equation}
where $H$ is an $n\times m$ matrix. Entry-wise this relation is equivalent to
	\begin{equation}
		[X_{ij},Y_{kl}]=-(\delta_{ik}H_{jl}-\delta_{jk}H_{il}), \quad \quad 1\leq i,j,k \leq n, 1\leq l \leq m.
	\end{equation}
Moreover, for $n=2,3 $,  we assume that
	\begin{align}\label{antisym-rela111}
&  X_1H_2=H_2X_1.
	\end{align}
Then in $ \mathcal A \ot \mathrm{Hom}( (\mathbb C^m)^{\otimes r},  (\mathbb C^n)^{\otimes r})$ one has that
\begin{equation}\label{antisym-turnidentity}
	S_r(X_1Y_1+(r-1)H_1)\cdots (X_rY_r) =	S_r X_1\cdots X_r Y_1\cdots Y_r.
\end{equation}
 for any $r, n, m$.
	In particular,
	\begin{equation}\label{per anti}
		{\per}\big((XY)_{IL}+H_{IL} \diag(r-1,\ldots,0) \big) \\
		=\sum\limits_{J}\dfrac{1}{v(\alpha_{J})} {\per}(X_{IJ}){\per}(Y_{JL}),
	\end{equation}
where $I=\{i_1,\ldots ,i_r\} \subset[n]$ and $L=\{l_1,\ldots,l_r\}  \subset[m]$ are two sets with cardinality $r$, and the sum is over all indices of non-decreasing integers $J=(j_1,\ldots,j_r)$ in $ [n]$.
\end{theorem}
\begin{proof}
	Equation \eqref{per anti} follows from \eqref {antisym-turnidentity}.
	By the argument in the proof of Theorem \ref{sym-turnbull}, we have that $X_1H_2=H_2X_1$ for $n\geq 4$.
	We then prove equation \eqref{antisym-turnidentity} by induction on $r$.
	The case  $r=1$ is clear. Let
	$S_{r-1}$  be the symmetrizer on the indices $\{2,\ldots,r\}$. Using the relation $	S_r=	S_r	S_{r-1}$ and  the induction hypothesis, we have that
	\begin{equation}
		\begin{aligned}
			&S_r(XY+(r-1)H)_1\cdots (XY)_r \\
			&=S_r(XY+(r-1)H)_1 S_{r-1}(XY+(r-2)H)_2\cdots (XY)_r\\
			&=S_r(XY+(r-1)H)_1 S_{r-1}X_2\cdots X_r Y_2\cdots Y_r\\
			&=S_r(XY+(r-1)H)_1X_2\cdots X_r Y_2\cdots Y_r.
		\end{aligned}
	\end{equation}
Applying the relations \eqref{antisym-rela1}, we have that
\begin{align*}
		&S_r(XY+(r-1)H)_1X_2\cdots X_r Y_2\cdots Y_r\\
		&=S_r X_1\cdots X_r Y_1\cdots Y_r\\
		&\quad +S_r\sum_{i=2}^{r}X_1X_2\cdots X_{i-1}(Q^{(1,i)}H_1-P^{(1,i)}H_1)X_{i+1}\cdots X_rY_2\cdots Y_r\\
		&\quad + (r-1) S_rH_1X_2\cdots X_rY_2\cdots Y_r .
\end{align*}
For $2\leq i \leq r$,  we have
\begin{equation}\label{eq antisym}
	\begin{aligned}
		&S_rX_1X_2\cdots X_{i-1}P^{(1,i)}H_1\\
		&=S_rH_1X_2\cdots X_{i-1}X_i.
	\end{aligned}
\end{equation}
Combining \eqref{eq antisym}  with Lemma \ref{AXQ-SXQ}, we deduce that
	\begin{equation}
		\begin{aligned}
			&S_r\sum_{i=2}^{r}X_1X_2\cdots X_{i-1}(Q^{(1,i)}H_1-P^{(1,i)}H_1)X_{i+1}\cdots X_rY_2\cdots Y_r\\
			&\quad + (r-1) S_rH_1X_2\cdots X_rY_2\cdots Y_r \\
&=0.
		\end{aligned}
	\end{equation}
This completes the proof.
\end{proof}

\begin{remark} Similar to the case of symmetric matrices, Theorem \ref{antisym-turnbull} specializes to Turnbull's identity  for anti-symmetric matrices \cite{UT,Turnbull} when $H=I$. Theorem \ref{antisym-turnbull}
recovers Turnbull's identity for anti-symmetric matrices given in \cite{FZ} and \cite{CSS,CFR} when $H$ is the diagonal matrix $hI$.
\end{remark}

\section{General Howe-Umeda-Kostant-Sahi identity}\label{s:gen-HUKS-identity}
In this section, we study the generalized Howe-Umeda-Kostant-Sahi identities for antisymmetric matrices which were conjectured by Caracciolo, Sokal and Sportiello \cite{CSS}.

The following theorem is a generalization of Howe-Umeda-Kostant-Sahi identity for antisymmetric matrices.
The special case of our identity ($h=1$, and $X$ is antisymmetric) recovers the classical Howe-Umeda-Kostant-Sahi identity  given in \cite{HU} and \cite{KS}.
\begin{theorem}\label{huks}
For even $n$, let $X=(X_{ij}), Y=(Y_{ij})\in \Mat_{n}(\mathcal A)$ with commutative entries.
Assume $Y$ is antisymmetric, and in  $\mathcal A\otimes  \mathrm{End}((\mathbb C^{n})^{\ot 2})$
	\begin{equation}\label{huks-rela1}
		Y_1X_2-X_2Y_1=h(P-Q),
	\end{equation}
where $h$ commutes with all entries of $X$ and $Y$. This relation can be written in term of entries as
	\begin{equation}
		[X_{ij},Y_{kl}]=-h(\delta_{jk}\delta_{il}-\delta_{ik}\delta_{jl}), \quad \quad 1\leq i,j,k,l \leq n.
	\end{equation}
	Then
	\begin{equation}\label{huks-identity}
		{\det}\big(XY+h \diag(n-2,n-3,\ldots,0,-1) \big) \\
		= {\det}(X){\det}(Y).
	\end{equation}
\end{theorem}

For the case of odd $n$, we have the following new Howe-Umeda-Kostant-Sahi type identity.

\begin{theorem}\label{odd-anti}
Let $n$ be an odd positive
	integer, and let $X=(X_{ij}), Y=(Y_{ij})\in \Mat_{n}(\mathcal A)$ both with commutative entries.
	Suppose either
$X$ or $Y$ is antisymmetric, and in  $\mathcal A \otimes  \mathrm{End}((\mathbb C^{n})^{\ot 2})$
	\begin{equation}\label{odd-anti-rela1}
		Y_1X_2-X_2Y_1=h(P-Q),
	\end{equation}
where $h$ commutes with all entries of $X$ and $Y$.
	In terms of entries this relation can be written as
	\begin{equation}
		[X_{ij},Y_{kl}]=-h(\delta_{jk}\delta_{il}-\delta_{ik}\delta_{jl}), \quad \quad 1\leq i,j,k,l \leq n.
	\end{equation}
	Then one has the following identity
      \begin{equation}\label{odd-anti-identity1}
		{\det}\big(XY+h \diag(n-1,n-2,\ldots,1,0) \big) \\
		= {\det}(X){\det}(Y)=0.
	 \end{equation}
\end{theorem}
These two theorems are proved in the following subsections.

\subsection{Identities for Pfaffians}
In this section we recall some identities for Pfaffians which will be used to  prove  Theorems \ref{huks} and \ref{odd-anti}.

Let $X$ be an $2m\times 2m$ antisymmetric matrix with commutative entries.
The  Pfaffian of  $X$ is defined by
\begin{align*}
	\Pf(X)
	=\frac{1}{2^mm!}\sum_{\sigma\in \mathfrak{S}_{2m}}(-1)^{l(\sigma)}x_{\sigma(1)\sigma(2)}x_{\sigma(3)\sigma(4)}\cdots x_{\sigma(2m-1)\sigma(2m)}.
\end{align*}

Let $I=\{i_1,i_2,\ldots,i_{2r}\}$ be a subset of $\{1,2,\ldots,2m\}$ with $i_1<i_2<\ldots<i_{2r}$. Denote the complement of $I$ by $I^c$. And
denote by $X_I$ the matrix obtained from $X$ by picking up the rows and columns indexed by $I$. We denote the Pfaffian of $X_I$ by
$\Pf (X_I)=(i_1,i_2,\dots,i_{2r})$.

%

Denote by $\Pi_m$ the set of $2$-shuffles, consisting of all $\sigma$ in $\mathfrak{S}_{2m}$ such that
$\sigma_{2k-1}<\sigma_{2k}$, $1 \leq k\leq m $ and $\sigma_{1}<\sigma_{3}<\ldots <\sigma_{2m-1}$.
For any set $I=\{i_1,i_2,\ldots,i_{2r}\}$ with $i_1<i_2<\ldots<i_{2r}$,
we denote by $\Pi_{I}$ the set of $2$-shuffles on the indices $\{i_1,i_2,\dots,i_{2r} \}$.

Then  Pfaffian can be computed by
\begin{equation}
	\begin{split}
		\Pf (X)&=\sum_{\pi \in\Pi_m}(-1)^{l(\pi)}(i_1,j_1)(i_2,j_2)\cdots(i_m,j_m)\\
		&=\sum_{j=2}^{2m}(-1)^{j-2}(1,j)(2,3,\ldots,\hat{j},\ldots,2m),
	\end{split}
\end{equation}
where $\pi_{2k-1}=i_{k}$ and $\pi_{2k}=j_{k}$ for $1 \leq k \leq m$ and $(2,3,\ldots,\hat{j},\ldots,2m)$ obtained from
$(2,3,\ldots,2m)$ by deleting $j$.

We now recall the Laplace-type expansion formulas for Pfaffian \cite{Cai,SO}.
\begin{theorem}\label{Laplace formulas for Pfaffian}
	Let $m$ and $n$ be nonnegative integers such that $m+n$ is even. For an $m\times m$ antisymmetric matrix $Z=(z_{ij})$, an $n\times n$ antisymmetric matrix $Z'=(z'_{ij})$, and an $m\times n$ matrix $W=(w_{ij})$, we have
	\begin{equation}
		\Pf\begin{pmatrix}
			Z & W\\
			-W^{t} & Z'
		\end{pmatrix}
		= \sum_{I,J} \varepsilon(I,J)\Pf(Z_{I})\Pf(Z'_{J})\det(W_{[m]\setminus I,[n]\setminus J}),
	\end{equation}
	where the sum is taken over all pairs of even-element subsets $(I,J)$ such that $I \subset [m], J \subset [n]$ and $m-\mid I\mid = n-\mid J \mid $, and the coefficient $\varepsilon(I,J)$ is given by
	\begin{equation*}
		\varepsilon(I,J)=(-1)^{\sum(I)+\sum(J)+\binom{m}{2}+\binom{n}{2}+\binom{k}{2}}, \ k=m-\mid I\mid = n-\mid J \mid,
	\end{equation*}
where we denote $\sum(I)=\sum_{i \in I} i$.
\end{theorem}
\begin{corollary}
	Suppose that $m+n$ is even, if $Z=(z_{ij})$ be an $m\times m$ antisymmetric matrix and  $W=(w_{ij})$ is an $m\times n$ matrix , we have
	\begin{equation}
		\Pf\begin{pmatrix}
			Z & W\\
			-W^{t} & 0
		\end{pmatrix}
		= \left\{ \begin{aligned}
			&\sum_{I} (-1)^{\sum(I)+\binom{m}{2}}\Pf(Z_{I})\det(W_{[m]\setminus I,[n]})\ &if\ m>n,\\
			&(-1)^{\binom{m}{2}}\det(W) \ &if\ m=n,\\
			&0 \ &if\ m<n,\\
		\end{aligned} \right.
	\end{equation}
	where $I$ runs over all $(m-n)$-element subsets of $[n]$.
\end{corollary}

If $X$ is an $n \times n$ antisymmetric matrix and $U$ is an $n \times n$ matrix, then
\begin{equation}
	\Pf(U^{t}XU)=\det(U)\Pf(X).
\end{equation}

For any $\si \in \mathfrak{S}_n$, one has that
\begin{equation}\label{permutation Pf}
 \Pf((X_{\si(i),\si(j)})_{1\leq i,j \leq n})=sgn(\si)\Pf((X_{i,j})_{1\leq i,j \leq n}).
\end{equation}

We denote $e_{i_1} \ot \cdots \ot e_{i_{2m}}$ by $e_{i_1,i_2, \ldots  ,i_{2m}}$.
Recalling the transposition $P$ \eqref{e:P} and transposed transposition $Q$ \eqref{e:Q}, it is easy to see that
\begin{equation}\label{eq pf}
	\begin{aligned}
		&\sum_{\sigma\in \Pi_{m}} sgn(\sigma)Q^{(1,2)}\cdots   Q^{(2m-1,2m)}X_{2}\cdots X_{2m} P^{\si^{-1}}e_{i_1,i_2, \ldots  ,i_{2m}} \\
		&=\Pf(X_{I})\sum_{I'} 
		e_{i'_1,i'_1,i'_3,i'_3, \ldots ,i'_{2m-1}, i'_{2m-1}},
	\end{aligned}
\end{equation}
where the sum is over all $I'=(i'_1,i'_3, \ldots  ,i'_{2m-1})$ with $i'_{2k-1} \in [n]$ for any $1 \leq k \leq m$.

We denote
\begin{equation}
	F_{m}=\sum_{\sigma\in \Pi_{m}} sgn(\sigma) P^{\si^{-1}}.
\end{equation}
Then \eqref{eq pf} can be written as

\begin{equation}
	\begin{aligned}
		& Q^{(1,2)}\cdots   Q^{(2m-1,2m)}X_{2}\cdots X_{2m}  F_m e_{i_1,i_2, \ldots  ,i_{2m}} \\
		&=\Pf(X_{I}) \sum_{I'}
		e_{i'_1,i'_1,i'_3,i'_3, \ldots , i'_{2m-1}, i'_{2m-1}}.
	\end{aligned}
\end{equation}


We denote $e_{I,J,\ldots,K}=e_{i_1, \ldots  , i_{r},j_1,\ldots,j_s,\ldots,k_1,\ldots,k_t}$ for any sets  $I=\{i_1<i_2<\ldots<i_r\}$,  $J=\{j_1<j_2<\ldots<j_s\}$ and $K=\{k_1<k_2<\ldots<k_t\}$.
We also denote by  $inv(I,J,\ldots, K)$  the inversion number of $\{I,J,\ldots,K\}$ and denote its sign by
$sgn(I,J,\ldots,K)=(-1)^{inv(I,J,\ldots,K)}$.

Let
\begin{equation}
	G_{m}e_{1,2,\ldots,n}=\sum_{I \subset [n], \atop \mid I \mid =2m } sgn(I,I^c) e_{I,I^c}.
\end{equation}
In fact, for any $I=\{i_1<i_2<\ldots<i_{2m-1}<i_{2m}\} \subset [n]$, there exists only one permutation $\tau_{I} \in  \mathfrak{S}_{n}$ such that $\tau^{-1}(k)=i_k$  for any $1 \leq k \leq 2m$ and $\tau^{-1}(2m+1)<\tau^{-1}(2m+2)<\ldots<\tau^{-1}(n-1)<\tau^{-1}(n)$. Then we have that
\begin{equation}
	G_{m}=\sum_{I \subset [n], \atop \mid I \mid =2m } sgn(\tau_{I}) P^{\tau_{I}}.
\end{equation}

The following lemma will be used to prove Lemma  $\ref{huks-lemma1}$ and   \ref{odd-claim}  .
\begin{lemma}\label{claim-lemma} For positive integer $n$, let 
$X=(X_{ij})\in \Mat_{n }(\mathcal A)$, then in $  \mathcal A\ot \mathrm{End}( (\mathbb C^n)^{\otimes n})$
	\begin{equation}\label{operator1}
		\begin{aligned}
			A_{3}(X_iQ^{(i,j)}Q^{(j,k)}-X_jQ^{(j,k)})=0,
		\end{aligned}
	\end{equation}
	for any $1\leq i<j<k\leq n$, where $A_3$ is  the antisymetrizer on the indices $\{i,j,k\}$.
	
	Equivalently,
	\begin{equation}\label{operator2}
		\begin{aligned}
			A_3(X_iQ^{(i,k)}Q^{(j,k)}-X_kQ^{(j,k)})=0.
		\end{aligned}
	\end{equation}
\end{lemma}
\begin{proof}
	The equation \eqref{operator2} follows from  \eqref{operator1} by relation $P^{(j,k)}Q^{(j,k)}=Q^{(j,k)}$. For any $1\leq i<j<k\leq n$, we have that
	\begin{equation}\label{lemm-term1}
		\begin{aligned}
			&A_3(X_iQ^{(i,j)}Q^{(j,k)}-X_jQ^{(j,k)})\\
			&=P^{(3,k)}P^{(2,j)}P^{(1,i)}A_3(X_1Q^{(1,2)}Q^{(2,3)}-X_2Q^{(2,3)})P^{(1,i)}P^{(2,j)}P^{(3,k)},
		\end{aligned}
	\end{equation}
where  $A_3$ is  the antisymetrizer on the indices $\{1,2,3\}$.
	So for \eqref{operator1}, it is sufficient to show that
	\begin{equation}\label{operator3}
		A_3(X_1Q^{(1,2)}Q^{(2,3)}-X_2Q^{(2,3)})=0,
	\end{equation}
which can be checked by showing that
	\begin{equation}
		A_3(X_1Q^{(1,2)}Q^{(2,3)}-X_2Q^{(2,3)})e_{k_1,k_2,k_3}=0
	\end{equation}
	for any $k_1,k_2,k_3$.
	
	In fact, if $k_2\neq k_3$, then $Q^{(2,3)}e_{k_1,k_2,k_3}=0$.
	If $k_2= k_3$, we have that
	\begin{equation}
		\begin{aligned}
			&A_3(X_1Q^{(1,2)}Q^{(2,3)}-X_2Q^{(2,3)})e_{k_1,k_2,k_2}\\
			&=A_3X_1Q^{(1,2)}e_{k_1,k_1,k_1}-A_3X_2\sum_{l\neq       k_1}e_{k_1,l,l}\\
			&=A_3X_1\sum_{k \neq k_1}e_{k,k,k_1}-A_3P^{(2,3)}P^{(1,2)}X_2\sum_{l\neq k_1}e_{k_1,l,l}\\
			&=A_3X_1\sum_{k \neq k_1}e_{k,k,k_1}-A_3X_1\sum_{  l\neq k_1}e_{l,l,k_1}\\
			&=0.
		\end{aligned}
	\end{equation}	
	Therefore equation \eqref{operator3} is verified, and the lemma is proved.
\end{proof}

In the following, we denote $\overline{[k]}=\{1,3,\dots,2k-1\}$ and $\underline{[k]}=\{2,4,\dots,2k\}$ for any $k \geq 1$. We denote by $X^I_J$ the submatrix of $X$ whose row (resp. column) indices belong to  $I$ (resp. $J$), and $\det(X^I_J)$ is the determinant of $X^I_J$.

For any set  $I=\{i_1<i_2<\ldots<i_{2m}\}$ and  $\pi \in \Pi_m$, we denote $I'_{\pi}=\{i_{\pi_1}, i_{\pi_3}, \ldots ,  i_{\pi_{2m-1}}\}$ and $I''_{\pi}=\{i_{\pi_2}, i_{\pi_4}, \ldots , i_{\pi_{2m}}\}$.

\subsection{Proof of Theorem \ref{huks}}
Equation \eqref{huks-identity} is equivalent to the following equation
\begin{equation}\label{huks-identity1}
	\begin{aligned}
		&A_n(X_1Y_1+ h(n-2))\cdots (X_{n-1}Y_{n-1}) (X_{n}Y_{n}-h)e_{1,\ldots ,n}\\
		&=A_n X_1\cdots X_{n} Y_1\cdots Y_ne_{1,\ldots ,n}.
	\end{aligned}
\end{equation}

The following Lemma gives an expansion of the left side of Eq. \eqref{huks-identity1}.
\begin{lemma}\label{huks-lemma1}Let $n$ be even, matrices $X$ and $Y$ satisfy the conditions in Theorem \ref{huks}. For any $m\leq n$, one has
in  $  \mathcal A \ot \mathrm{End}( (\mathbb C^n)^{\otimes m})$
\begin{equation}\label{huks-claim1}
	\begin{aligned}
		&A_m(X_1Y_1+ h(m-2))\cdots (X_{m}Y_{m}-h)\\
		&=A_m X_1\cdots X_{m} Y_1\cdots Y_m\\
		&\quad +A_m\sum_{p=0}^{\lceil m/2 \rceil-1} (-h)^{p+1}(s_p+w_p),
	\end{aligned}
\end{equation}
where  $\lceil m/2 \rceil$ is the least integer $\geq m/2$ and
\begin{equation}\label{sp-term}
\begin{aligned}
		s_p = \sum_{\mid I\mid=2p+2, \atop \pi \in \Pi_{I} }  \left(\prod_{i \in  I^{''c}_{\pi}}^m X_{i}\right) Q^{(i_{\pi_1},i_{\pi_2})}\cdots Q^{(i_{\pi_{2p+1}},i_{\pi_{2p+2}})} \left(\prod_{i \in   I^{'c}_{\pi} }^m Y_{i}\right),
\end{aligned}
\end{equation}
where the sum is taken over all $I=\{i_1<i_2<\ldots<i_{2p+2}\}\subset [1, m]$ 
and all $\pi \in \Pi_{I}$.
\begin{equation}\label{wp-term}
	\begin{aligned}
		w_p= \sum_{\mid I\mid=2p, \atop \si \in\Pi_{I} }  \sum_{k \in I^c} \left(\prod_{ i \in  I^{''c}_{\si} \setminus \{k\} }^m X_{i}\right)Q^{(i_{\si_1},i_{\si_2})}\cdots Q^{(i_{\si_{2p-1}}, i_{\si_{2p}})} \left(\prod_{ i \in I^{'c}_{\si} \setminus \{k\} }^m Y_{i}\right).
	\end{aligned}
\end{equation}

\end{lemma}
\begin{proof}
We prove the identity \eqref{huks-claim1} by induction on $m$. It is obvious for $m=1$. Note that
\begin{equation}
	\begin{aligned}
		&A_m(X_1Y_1+ h(m-2))(X_2Y_2+h(m-3))\cdots (X_{m}Y_{m}-h)\\
		&=A_m(X_1Y_1+ h(m-2)) A_{m-1} (X_2Y_2+h(m-3)) \cdots (X_{m}Y_{m}-h),
	\end{aligned}
\end{equation}
where  $A_{m-1}$ is the antisymmetrizer on the indices $\{2,\ldots,m\}$.

By the induction hypothesis, we have that
\begin{equation}
	\begin{aligned}
		&A_{m-1} (X_2Y_2+h(m-3)) \cdots (X_{m}Y_{m}-h)\\
		&=A_{m-1} X_2\cdots X_{m} Y_2\cdots Y_m\\
		&\quad +A_{m-1} \sum_{p=0}^{\lceil (m-1)/2 \rceil-1} (-h)^{p+1}(s'_p+w'_p),
	\end{aligned}
\end{equation}
where $s'_p$ and $w'_p$ are defined on the indices $\{2,\ldots,m\}$ in the same way as the above equations \eqref{sp-term} and \eqref{wp-term}.
	
Using relation \eqref{huks-rela1} and the following equations
\begin{equation*}
	\begin{aligned}
		&A_mX_1\cdots X_{j-1}P^{(1,j)}=-A_mX_2\cdots X_{j},\ 2 \leq j \leq m,\\
		&A_m  \left(\prod_{i \in  (I^{''}_{\pi}\cup \{j\})^c } X_{i}\right)
		P^{(1,j)} =-A_m \left(\prod_{i \in  [2,m] \setminus I^{''}_{\pi} } X_{i}\right),\ j \in  [2,m] \setminus I^{''}_{\pi},
	\end{aligned}
\end{equation*}
we have that
\begin{equation}\label{huks-claim-term1}
	\begin{aligned}
		&A_m(X_1Y_1+ h(m-2))X_2\cdots X_{m}Y_2\cdots Y_m\\
		&=A_mX_{1}\cdots X_{m}Y_{1}\cdots Y_{m}-hA_mX_2\cdots X_{m}Y_2\cdots Y_m\\
		&\quad-hA_m \sum_{j=2}^{m} \left(\prod_{i \neq j } X_{i}\right)Q^{(1,j)}\left(\prod_{i \neq 1 } Y_{i}\right),
	\end{aligned}
\end{equation}
and
\begin{equation}\label{huks-claim-term2}
	\begin{aligned}
		&A_m(X_1Y_1+ h(m-2)) s'_p\\
		&=A_mX_1s'_pY_1+hpA_ms'_p -hA_m \sum_{I\subset [2,m] ,\atop \mid I \mid =2p+2}\sum_{\pi \in \Pi_{I}, \atop j \in [2,m]\setminus  I^{''}_{\pi}}  \left(\prod_{i \in   ( I^{''}_{\pi}\cup \{j\})^c } X_{i}\right)\\
		&\quad \times Q^{(1,j)} Q^{(i_{\pi_1},i_{\pi_2})}\cdots
	    Q^{(i_{\pi_{2p+1}} , i_{\pi_{2p+2}})} \left(\prod_{i \in  [2,m] \setminus I'_{\pi} }Y_{i}\right).
	\end{aligned}
\end{equation}
By  Lemma \ref{claim-lemma} it follows that for any  $j \in I'_{\pi} \subset [2,m]\setminus  I^{''}_{\pi} $,
\begin{equation}\label{huks-claim-term3}
	\begin{aligned}
		&A_m  \left(\prod_{i \in  ( I^{''}_{\pi}\cup \{j\})^c } X_{i}\right) Q^{(1,j)} Q^{(i_{\pi_1},i_{\pi_2})} \cdots
		Q^{(i_{\pi_{2p+1}} , i_{\pi_{2p+2}})} \\
		&=A_m \left(\prod_{i \in  [2,m] \setminus I^{''}_{\pi} } X_{i}\right)Q^{(i_{\pi_1},i_{\pi_2})} \cdots
		Q^{(i_{\pi_{2p+1}} , i_{\pi_{2p+2}})},
	\end{aligned}
\end{equation}
so combining identities \eqref{huks-claim-term2} and \eqref{huks-claim-term3}, we have that
\begin{equation}
	\begin{aligned}
		A_m(X_1Y_1+ h(m-2)) s'_p=A_mX_1s'_pY_1-hA_ms'_p-hA_mf_p,
	\end{aligned}
\end{equation}
where
\begin{equation}
   \begin{aligned}		
     	f_p&=\sum_{I\subset [2,m] ,\atop \mid I \mid =2p+2}\sum_{\pi \in \Pi_{I}, \atop j \in [2,m]\setminus I}  \left(\prod_{ i \in   ( I^{''}_{\pi} \cup \{j\})^c } X_{i}\right) Q^{(1,j)} Q^{(i_{\pi_1},i_{\pi_2})}\cdots\\
		&\quad \cdots Q^{(i_{\pi_{2p+1}} , i_{\pi_{2p+2}})} \left(\prod_{i \in  [2,m] \setminus I'_{\pi} }Y_{i}\right).
	\end{aligned}
\end{equation}	
Similarly, we can obtain that
\begin{equation}
	\begin{aligned}
		&A_m(X_1Y_1+ h(m-2)) w'_p=A_mX_1w'_pY_1-hA_mg_p,
	\end{aligned}
\end{equation}
where
\begin{equation}
	\begin{aligned}			   	
		g_p&= \sum_{I \subset [2,m],\atop  \mid I \mid =2p} \sum_{\substack{\si \in\Pi_{I}, \\ k, j \in [2,m] \setminus I,\\ k \neq j}}
		\left(\prod_{ i \in   ( I^{''}_{\si} \cup \{k,j\})^c } X_{i}\right) Q^{(1,j)} Q^{(i_{\si_1},i_{\si_2})}  \cdots \\
		&\quad \cdots Q^{(i_{\si_{2p-1}}, i_{\si_{2p}})} \left(\prod_{  i \in   ( I^{'}_{\si} \cup \{1,k\})^c} Y_{i}\right).
	\end{aligned}
\end{equation}

Subsequently we have that
\begin{equation}\label{huks-claim-term5}
	\begin{aligned}
		&A_m(X_1Y_1+ h(m-2))(X_2Y_2+h(m-3))\cdots (X_{m}Y_{m}-h)\\
		&=A_mX_{1}\cdots X_{m}Y_{1}\cdots Y_{m}\\
		&\quad -hA_m  X_2\cdots X_{m}Y_2\cdots Y_m
		-hA_m \sum_{j=2}^{m}  \left(\prod_{i \neq j } X_{i}\right)Q^{(1,j)}\left(\prod_{i \neq 1 } Y_{i}\right)\\
		&\quad +A_m\sum_{p=0}^{\lceil \frac{(m-1)}2 \rceil-1} (-h)^{p+1}(X_1s'_pY_1-hs'_p-hf_p+X_1w'_pY_1-hg_p),
	\end{aligned}
\end{equation}
which implies \eqref{huks-claim1} by using
\begin{equation*}
	\begin{aligned}
		s_0&=X_1s'_0Y_1+\sum_{j=2}^{m} \left(\prod_{i \neq j } X_{i}\right)Q^{(1,j)}\left(\prod_{i \neq 1 } Y_{i}\right),\\
		w_0&=X_1w'_0Y_1+X_2\cdots X_{m}Y_2\cdots Y_m,
	\end{aligned}
\end{equation*}
for  $p \geq 1$,
\[
s_p=X_1s'_pY_1-hf_{p-1},\
w_p=-hs'_{p-1}+X_1w'_pY_1-hg_{p-1}.
\]
\end{proof}

In view of Lemma \ref{huks-lemma1}, in order to show \eqref{huks-identity1} it is sufficient to prove that
\begin{equation}\label{huks-identity2}
	\begin{aligned}
		  A_n(s_p+w_p)e_{1,\ldots,n}=0,
	\end{aligned}
\end{equation}
for any $0 \leq p \leq \dfrac{n}{2}-1$.
For this we first note that 
\begin{equation}
	\begin{aligned}
		&A_ns_p\\
&=A_n \left(\prod_{i \in \underline{[p+1]}^c} X_{i}\right) Q^{(1,2)} \cdots Q^{(2p+1,2p+2)}\left(\prod_{ i \in \overline{[p+1]}^c}Y_i\right) F_{p+1} G_{p+1},
	\end{aligned}
\end{equation}
and
\begin{equation}
	\begin{aligned}
		A_nw_p&=A_n\left(\prod_{i \in \underline{[p+1]}^c} X_{i}\right) Q^{(1,2)} \cdots Q^{(2p-1,2p)}\left( \prod_{ i \in \overline{[p]}^c \setminus \{2p+2\} }  Y_i\right) \\
		& \quad \times
		\left(1-\sum_{ k \notin ([2p] \cup \{2p+2\})}  P^{(k,2p+2)}\right) F_{p}G_{p}.
	\end{aligned}
\end{equation}

Therefore, $A_ns_pe_{1,\ldots,n}$ and $A_nw_pe_{1,\ldots,n}$ can be written as follows
\begin{equation*}
	\begin{aligned}
		&A_n s_p e_{1,\ldots,n}
		 = A_n \left(\prod_{i \in \underline{[p+1]}^c} X_{i}\right)A'_{p+1}A''_{n-p-1}  Q^{(1,2)}\cdots Q^{(2p-1,2p)} \varPhi  e_{1, \ldots ,n},\\
		&A_n w_p e_{1,\ldots,n}
		 = A_n \left(\prod_{i \in \underline{[p+1]}^c} X_{i}\right)A'_{p+1}A''_{n-p-1}  Q^{(1,2)}\cdots Q^{(2p-1,2p)}  \varPsi e_{1, \ldots ,n}.
	\end{aligned}
\end{equation*}
where $A'_{p+1}$ and $A''_{n-p-1}$ are  the antisymmetrizers on the indices  $\underline{[p+1]}$ and $[n] \setminus \underline{[p+1]}$ respectively, and
\begin{equation}\label{e:Phi}
	\begin{aligned}		
		\varPhi=Q^{(2p+1,2p+2)}\left(\prod_{ i \in \overline{[p+1]}^c}Y_i\right) F_{p+1} G_{p+1},
	\end{aligned}
\end{equation}
\begin{equation}\label{e:Psi}
	\begin{aligned}		
		\varPsi=\left( \prod_{ i \in \overline{[p]}^c \setminus \{2p+2\} }  Y_i\right)  \left(1-\sum_{ k \notin ([2p] \cup \{2p+2\})}  P^{(k,2p+2)}\right) F_{p}G_{p}.
	\end{aligned}
\end{equation}

The following Lemma implies \eqref{huks-identity2}, therefore Theorem \ref{huks} is proved.

\begin{lemma}\label{huks-lemma2} For even $n$, let
$Y$ be an anti-symmetric matrix in $\Mat_{n }(\mathcal A)$ with commuting entries.
Then for $0\leq p \leq \dfrac{n}{2}-1$,
\begin{equation}\label{husk-lemma-identity}
\begin{aligned}
	A'_{p+1}A''_{n-p-1}  Q^{(1,2)}\cdots Q^{(2p-1,2p)}(\varPhi + \varPsi)e_{1, \ldots ,n}=0,
\end{aligned}
\end{equation}
where $Q$ is the transposed transposition, $\varPhi=\varPhi(Y)$ and $\varPsi=\varPsi(Y)$ are defined in \eqref{e:Phi} and \eqref{e:Psi}.
\end{lemma}
\begin{proof}
	The coefficient of $e_k\ot e_l$ in $Q\cdot(e_i\ot e_i)$ is $\delta_{kl}$, then we have that
\begin{equation}\label{huks-lemma-term1}
    \begin{aligned}
		&A'_{p+1}A''_{n-p-1}  Q^{(1,2)}\cdots Q^{(2p-1,2p)}\varPhi e_{1, \ldots ,n}\\
		&=A'_{p+1}A''_{n-p-1} \sum_{\mid I \mid =2p+2, \atop \mid I'\sqcup J \mid =n-p-1}  sgn(I,I^c) \Pf(Y_{I }) \det(Y_{I ^c } ^{J}) e_{i'_1,i'_1,i'_3,i'_3,\ldots,i'_{2p+1},i'_{2p+1} , J} \\	
        &=A'_{p+1}A''_{n-p-1}\sum_{ t \in [n] }\sum_{\mid I \mid =2p+2,\atop I'\sqcup J \subseteq [n]\setminus \{t\}} sgn(I,I^c) \Pf(Y_{I})\det(Y_{I^c}^{J}) e_{i'_1,i'_1,i'_3,i'_3,\ldots,t,t , J}, 
		\end{aligned}
\end{equation}	
 where the sum is over all $I=\{1\leq i_1<i_2<\ldots<i_{2p+2}\leq n\}$
and $(n-p-2)$-element subsets $I'\sqcup J \subseteq [n]\setminus \{t\}$, such that $J=\{1\leq j_1<j_2<\ldots<j_{n-2p-2} \leq n\}$. Note that
\begin{align*}	
		&A'_{p+1}A''_{n-p-1}Q^{(1,2)}\cdots Q^{(2p-1,2p)}\varPsi e_{1, \ldots ,n}\\
		&=A'_{p+1}A''_{n-p-1} \sum_{\mid I' \mid =p \atop \mid I \mid =2p}  sgn(I,I^c) \Pf(Y_{I})  \left( \prod_{ i \notin ([2p]\cup\{2p+2\}) }  Y_i\right) \\
&\quad \times
\left(1-\sum_{ k \notin ([2p] \cup \{2p+2\})}  P^{(k,2p+2)}\right)
 e_{i'_1,i'_1,i'_3,i'_3,\ldots,i'_{2p-1},i'_{2p-1}, I^c}\\
        &=-A'_{p+1}A''_{n-p-1} \sum_{\mid I \mid =2p,\atop \mid I'\sqcup J\mid =n-p-1}  \sum_{k=2p+1} ^{n} sgn(I,i_k,I^c\setminus\{i_k\}) \Pf(Y_{I })\\
&\quad 
\det(Y_{ I^c \setminus \{i_k\}  } ^{J})
e_{i'_1,i'_1,i'_3,i'_3,\ldots,i'_{2p-1},i'_{2p-1}, j_1,i_k,J \setminus \{j_1\}}, 
  	\end{align*}
where the sum is over all $I=\{1\leq i_1<i_2<\ldots<i_{2p}\leq n\}$ and $(n-p-1)$-elements subsets  $I'\sqcup J$, such that $J=\{1\leq j_1<j_2<\ldots<j_{n-2p-1} \leq n\}$. 
Here $I^{c}=\{i_{2p+1}<i_{2p+2}<\ldots<i_n\}$.

 Note that $sgn(I,i_k,I^c\setminus\{i_k\})=sgn(i_k,I,I^c\setminus\{i_k\})$. So,
\begin{equation}\label{huks-lemma-term2}
	\begin{aligned}	
		&A'_{p+1}A''_{n-p-1} Q^{(1,2)}\cdots Q^{(2p-1,2p)}\varPsi e_{1, \ldots ,n}\\
		&=-A'_{p+1}A''_{n-p-1} \sum_{\substack{ t \in [n], \\ \mid J\mid =n-2p-1}}  \sum_{I' \subseteq [n] \setminus (\{t\} \cup J) }\sum_{\mid I\mid =2p,\atop  I \subset [n]\setminus \{t\}}  sgn(t,I,I^c\setminus\{t\})   \Pf(Y_{I})\\
&\quad \qquad\qquad\qquad  \cdot\det(Y_{ I^c \setminus \{t\}}^{J}) e_{i'_1,i'_1,i'_3,i'_3,\ldots,i'_{2p-1},i'_{2p-1}, j_1,t,J \setminus \{j_1\}}. 
	\end{aligned}
\end{equation}

If $t \notin J$ in \eqref{huks-lemma-term2}, we show that for any  subset $ J \subset [n]\setminus \{t\}$ with cardinality $n-2p-1$,
\begin{equation}\label{huks-lemma-term6}
	\begin{aligned}
		 \sum_{\mid I\mid =2p,\atop  I \subset [n]\setminus \{t\}}  sgn(I,I^c\setminus\{t\}) \Pf(Y_{I}) \det(Y_{ I^c \setminus \{t\}}^{J})=0.
	\end{aligned}
\end{equation}

Using the Laplace-type expansion for Pfaffian in lemma \ref{Laplace formulas for Pfaffian}, we have that
\begin{equation}
	\begin{aligned}
		&\sum_{\mid I\mid =2p,\atop  I \subset [n]\setminus \{t\}}  sgn(I,I^c\setminus\{t\}) \Pf(Y_{I}) \det(Y_{ I^c \setminus \{t\}}^{J})\\
		&=\sum_{\mid I\mid =2p,\atop  I \subset [n]\setminus \{t\}}  (-1)^{1+inv(I,I^c\setminus\{t\})}
		\Pf(Y_{I}) \det(Y^{ I^c \setminus \{t\}}_{J})\\
		&=(-1)^{p+1+\binom{n-1}{2}}
		\Pf\begin{pmatrix}   Y_{[n]\setminus\{t\}} & Y_{[n]\setminus \{t\},J}\\
			Y_{J,[n]\setminus \{t\}} & 0
		\end{pmatrix}.
	\end{aligned}
\end{equation}

Note that
\begin{equation*}
	\begin{aligned}
		\det\begin{pmatrix}   Y_{[n]\setminus\{t\}} & Y_{[n]\setminus \{t\},J}\\
			Y_{J,[n]\setminus \{t\}} & 0
		\end{pmatrix}=0,
	\end{aligned}
\end{equation*}
so \eqref{huks-lemma-term6} follows.

If  $t \in J$  in \eqref{huks-lemma-term2}, then
\begin{equation*}
	\begin{aligned}
            &A''_{n-2p-1}\det(Y_{ I^c \setminus \{t \}}^{J}) 
e_{i'_1,i'_1,i'_3,i'_3,\ldots,i'_{2p-1},i'_{2p-1}, j_1,t,J \setminus \{j_1\}} \\
            &=A''_{n-2p-1}\det(Y_{ I^c \setminus \{t \}}^{t,J \setminus \{t \}}) 
e_{i'_1,i'_1,i'_3,i'_3,\cdots,i'_{2p-1},i'_{2p-1}, t,t,J \setminus \{t\}}, 
\end{aligned}
\end{equation*}
where $A''_{n-2p-1}$ is the antisymetrizer on the indices  $[n] \setminus( [2p] \cup \{2p+2\})$.

Comparing  \eqref{huks-lemma-term1} with  \eqref{huks-lemma-term2},
to show  \eqref{husk-lemma-identity} it is sufficient to prove that for any given $t \in [n]$ and subset $ J \subset [n]\setminus \{t\}$ with cardinality
$(n-2p-2)$,
\begin{equation}\label{huks-lemma-term3}
	\begin{aligned}
	   &\sum_{\mid I\mid=2p+2 } sgn(I,I^c) \Pf(Y_{I }) \det(Y_{I^c}^{J})\\
	   &= \sum_{\mid I\mid =2p,\atop  I \subset [n]\setminus \{t\}}  sgn(t,I,I^c\setminus\{t\}) \Pf(Y_{I }) \det(Y_{ I^c \setminus \{t\}}^{t,J}).
	\end{aligned}
\end{equation}


Using the Laplace-type expansion of Pfaffian in Lemma \ref{Laplace formulas for Pfaffian}, the left side of \eqref{huks-lemma-term3} is
\begin{equation}\label{huks-lemma-term4}
	\begin{aligned}
		&\sum_{\mid I\mid=2p+2 } sgn(I,I^c) \Pf(Y_{I }) \det(Y_{I^c}^{J})\\
		&=\sum_{\mid I\mid=2p+2 } (-1)^{\sum(I)+(p+1)}\Pf(Y_{I})\det(Y^{I^c}_{J})\\
		&=(-1)^{(p+1)+\binom{n}{2}}
		\Pf\begin{pmatrix}   Y_{[n]} & Y_{[n],J}\\
			Y_{J,[n]} & 0
		\end{pmatrix}.
	\end{aligned}
\end{equation}

The right side  of  \eqref{huks-lemma-term3} is
\begin{equation}\label{huks-lemma-term5}
	\begin{aligned}
		&\sum_{\mid I\mid =2p,\atop  I \subset [n]\setminus \{t\}}  sgn(t,I,I^c\setminus\{t\}) \Pf(Y_{I }) \det(Y_{ I^c \setminus \{t\}}^{t,J})\\
	    &=\sum_{I}  (-1)^{t+inv(I,I^c\setminus\{t\})}  \Pf(Y_{I}) \det(Y^{ I^c \setminus  \{t\}}_{t,J})\\
	    &=(-1)^{t + \frac{2p(2p+1)}{2} +\binom{n-1}{2}} \Pf(A)\\
	    &=(-1)^{t+p+\binom{n-1}{2}} \Pf(A),
	\end{aligned}
\end{equation}
where
\begin{equation*}
	\begin{aligned}
A=\begin{pmatrix}   Y_{[n] \setminus \{t\}} & Y_{[n] \setminus \{t\} , J}  & Y_{[n] \setminus \{t\} ,t} \\
	Y_{J,[n]  \setminus \{t\} }  & 0 & 0\\
	Y_{t, [n] \setminus \{t\} }  & 0 & 0
\end{pmatrix}.
	\end{aligned}
\end{equation*}
It follows from identity \eqref{permutation Pf} that
\begin{equation}
	\begin{aligned}
		\Pf(A)=(-1)^{t} \Pf\begin{pmatrix}   Y_{[n]} & Y_{[n],J}\\
			Y_{J,[n]} & 0
		\end{pmatrix},
	\end{aligned}
\end{equation}
which then implies \eqref{huks-lemma-term3}.

\end{proof}

\subsection{Proof of Theorem \ref{odd-anti}}

Equation \eqref{odd-anti-identity1} is equivalent to the following equation:
\begin{equation}\label{odd-anti-identity2}
	\begin{aligned}
		&A_n(X_1Y_1+ h(n-1))\cdots X_{n}Y_{n}e_{1,\ldots,n},\\
		&=A_n X_1\cdots X_{n} Y_1\cdots Y_ne_{1,\ldots ,n}
	\end{aligned}
\end{equation}

The following Lemma gives an expansion of the left side of \eqref{odd-anti-identity2}.
\begin{lemma}\label{odd-claim}
Let $n$ be odd. The matrices $X$ and $Y$ satisfy the conditions of  Theorem \ref{odd-anti} , then for any $m\leq n$, the following identity holds
in $ \mathcal A \ot \mathrm{End}( (\mathbb C^n)^{\otimes m})$:
\begin{equation}\label{odd-claim-iden}
	\begin{aligned}
		&A_m(X_1Y_1+ h(m-1))\cdots X_{m}Y_{m}\\
		&=A_mX_1\cdots X_mY_1\cdots Y_m +A_m\sum_{k=1}^{\lceil (m-1)/2 \rceil }(-h)^k s_{k-1},
	\end{aligned}
\end{equation}
where
 \begin{equation}\label{sp-term2}
		\begin{aligned}
			s_{k-1} =\sum_{\mid I\mid=2k, \atop \pi \in \Pi_{k} } \left(\prod_{i \in  I^{''c}_{\pi}}^m X_{i}\right) Q^{(i_{\pi_1},i_{\pi_2})}\cdots Q^{(i_{\pi_{2k-1}},i_{\pi_{2k}})} \left(\prod_{i \in   I^{'c}_{\pi} }^m Y_{i}\right).
		\end{aligned}
	\end{equation}

\end{lemma}
\begin{proof} We prove the identity \eqref{odd-claim-iden} by induction on $m$ using a similar argument as
in Lemma \ref{huks-lemma1}.
It is obvious for $m=1$. For $m>1$,
 we have that
\begin{equation}\label{odd-claim-term2}
	\begin{aligned}
		&A_m(X_1Y_1+ h(m-1))(X_2Y_2+h(m-2))\cdots X_{m}Y_{m}\\
		&=A_m(X_1Y_1+ h(m-1)) A_{m-1} (X_2Y_2+h(m-2)) \cdots X_{m}Y_{m}\\
		&=A_m(X_1Y_1+ h(m-1))X_2\cdots X_mY_2\cdots Y_m \\
		&\quad +A_m\sum_{k=1}^{\lceil (m-2)/2 \rceil }(-h)^k (X_1Y_1+ h(m-1)) s'_{k-1},
	\end{aligned}
\end{equation}	
where  $s'_{k-1}$ is defined on the indices $\{2,\ldots,m\}$ as equation \eqref{sp-term2}.

It follows from \eqref{odd-anti-rela1} and Lemma \ref{claim-lemma} that
\begin{equation}\label{odd-claim-term3}
	\begin{aligned}	
		    &A_m(X_1Y_1+ h(m-1))X_2\cdots X_mY_2\cdots Y_m\\
			&=A_mX_{1}\cdots X_{m}Y_{1}\cdots Y_{m}\\ &\quad-hA_m\sum_{j=2}^{m} \left(\prod_{i \neq j }X_{i}\right) Q^{(1,j)}\left(\prod_{i \neq 1 }Y_{i}\right),
		\end{aligned}
\end{equation}
and
\begin{equation}
	\begin{aligned}	
		    &A_m\sum_{k=1}^{\lceil (m-2)/2 \rceil }(-h)^k (X_1Y_1+ h(m-1)) s'_{k-1}\\		
			&=-hA_mX_1s'_{0}Y_1 +A_m\sum_{k=2}^{\lceil (m-1)/2 \rceil } (-h)^k  (X_1s'_{k-1}Y_1-hf_{k-2})\\
			&=-hA_mX_1s'_{0}Y_1+A_m\sum_{k=2}^{\lceil (m-1)/2 \rceil } (-h)^k s_{k-1},
	\end{aligned}
\end{equation}
where
\begin{equation}
	\begin{aligned}		
			f_{k-2}&=\sum_{I\subset [2,m] ,\atop \mid I \mid =2k-2}\sum_{\pi \in \Pi_{k-1}, \atop j \in [2,m]\setminus I}  \left(\prod_{ i \in   ( I^{''}_{\pi} \cup \{j\})^c } X_{i}\right) Q^{(1,j)} Q^{(i_{\pi_1},i_{\pi_2})}\cdots\\
			&\quad \cdots Q^{(i_{\pi_{2k-3}} , i_{\pi_{2k-2}})} \left(\prod_{i \in  [2,m] \setminus I'_{\pi} }Y_{i}\right).
	\end{aligned}
\end{equation}	
Since
\begin{equation}
s_0=X_1s'_{0}Y_1+
\sum_{j=2}^{m} \left(\prod_{i \neq j }X_{i}\right) Q^{(1,j)}\left(\prod_{i \neq 1 }Y_{i}\right),
\end{equation}
we obtain the identity \eqref{odd-claim-iden}.
\end{proof}

Therefore, in order to show \eqref{odd-anti-identity2}, it is sufficient to prove that
\begin{equation}\label{odd-anti-identity3}
	\begin{aligned}
		&A_ns_{k-1}e_{1,\ldots,n}=0,
	\end{aligned}
\end{equation}
for any $1 \leq k \leq \dfrac{n-1}{2}$.
Note that $A_ns_{k-1}$ can be written as
\begin{equation}\label{odd-anti-identity4}
	\begin{aligned}
		A_ns_{k-1}=A_n\left(\prod_{i \in   \underline{[k]}^c} X_{i}\right) Q^{(1,2)} \cdots Q^{(2k-1,2k)}\left(\prod_{i \in   \overline{[k]}^c }Y_{i}\right) F_{k}G_{k}.
	\end{aligned}
\end{equation}

The following  lemma implies \eqref{odd-anti-identity3}, which implies \eqref{odd-anti-identity2}.  Subsequently Theorem \ref{odd-anti} follows.
\begin{lemma}\label{anti-odd-lemma}
For odd $n$, let $X$ be an antisymmetric matrix in $\Mat_{n}(\mathcal A)$ with commuting entries.
Then for any $\pi \in \Pi_{k}$, $1 \leq k \leq \dfrac{n-1}{2}$, the following holds in $  \mathcal A \ot \mathrm{End}( (\mathbb C^n)^{\otimes n})$:
	\begin{equation}\label{lemma-odd-operator1}
		A_n\left( \prod_{ i \in I^{''c}_{\pi}} X_i \right)Q^{(i_{\pi_1},i_{\pi_2})} \cdots
		Q^{(i_{\pi_{2k-1}} , i_{\pi_{2k}})}=0.
	\end{equation}

	If $Y$ is an antisymmetric matrix in $\Mat_{n}(\mathcal A)$ with commuting entries,  then for any  $1 \leq k \leq \dfrac{n-1}{2}$ one has that
	\begin{equation}\label{lemma-odd-operator2}
		\begin{aligned}
		A_{n-k}Q^{(1,2)}\cdots Q^{(2k-1,2k)}\left( \prod_{i \in \overline{[k]}^c} Y_i \right) F_{k}G_{k} e_{1, \ldots,n}=0,
		\end{aligned}
	\end{equation}
	where  $A_{n-k}$ is the antisymetrizer on the indices  $[n] \setminus \underline{[k]}$.
\end{lemma}
\begin{proof}
	$(1)$ Assume that $X$ is antisymmetric. There exists  $\sigma\in  \mathfrak{S}_n$ such that
	\begin{equation}
		\begin{aligned}
			&A_{n}\left( \prod_{ i \in I^{''c}_{\pi}} X_i \right)Q^{(i_{\pi_1},i_{\pi_2})} \cdots
			Q^{(i_{\pi_{2k-1}} , i_{\pi_{2k}})}\\
			&=sgn(\sigma)A_n \left( \prod_{i \in [n-k]} X_i \right)Q^{(n-2k+1,n-k+1)}\cdots Q^{(n-k,n)} P^{\sigma}.
		\end{aligned}
	\end{equation}
	
In order to prove \eqref{lemma-odd-operator1} it is sufficient to show that
	\begin{equation}\label{lemma-odd-sufficient}
		A_n \left( \prod_{i \in [n-k]} X_i \right)Q^{(n-2k+1,n-k+1)}\cdots Q^{(n-k,n)}e_{i_1,\dots,i_n}=0.
	\end{equation}
	
	Since,
	\begin{equation*}
		\begin{aligned}
			 A_n\left( \prod_{i \in [n-k]} X_i \right)=A_n\left( \prod_{i \in [n-k]} X_i \right)A'_{n-2k} ,
		\end{aligned}
	\end{equation*}
	where  $A'_{n-2k}$ be the antisymmetrizer on the indices $\{1,\ldots,n-2k\}$.
So it is sufficient to consider mutually distinct indices  $ i_1,\dots,i_{n-2k}$
and  $i_{n-2k+j}=i_{n-k+j}$ for $1\leq j \leq k$.
Therefore,
	\begin{equation}\label{odd-lemma-terms1}
		\begin{aligned}
			&A_n\left( \prod_{i \in [n-k]} X_i \right)Q^{(n-2k+1,n-k+1)}\cdots Q^{(n-k,n)}e_{i_1,\dots,i_n}\\
			&=A_n\left( \prod_{i \in [n-k]} X_i \right)  \dfrac{1}{2}\sum_{J} \left(e_{I, J, J}+e_{I, J^c, J^c} \right)\\
			&=\dfrac{1}{2}A_n  \sum_{J}\left(\det(X_{I, J}^{I, J^c})e_{I, J^c, J}+\det(X_{I, J^c}^{I, J})e_{I, J, J^c}\right) \\
			&=\dfrac{1}{2}A_n \sum_{J} \left( \det(X_{I, J}^{I, J^c})+  (-1)^{k}\det(X_{I, J^c}^{I, J})\right) e_{I, J^c, J}\\
			&=\dfrac{1}{2}A_n \sum_{J} \left( \det(X_{I, J}^{I, J^c})+(-1)^n  \det(X_{I, J}^{I, J^c}) \right)  e_{I, J^c, J},
		\end{aligned}
	\end{equation}
	where the sum is taken over all  subsets $J \subseteq [n]\setminus I$ with cardinality $k$.
	Since $n$ is odd, we obtain \eqref{lemma-odd-sufficient}.
	
	(2)  Assume that $Y$ is antisymmetric. We have that
	\begin{equation}
		\begin{aligned}
			&A_{n-k}Q^{(1,2)}\cdots Q^{(2k-1,2k)}\left( \prod_{i \in \overline{[k]}^c} Y_i \right) F_{k}G_{k} e_{1, \ldots,n}\\
			&=A_{n-k}\sum_{\mid I\mid =2k , \atop \mid I'\sqcup J\mid =n-k} sgn(I,I^c) \Pf(Y_{I})\det(Y_{I^c}^{J}) e_{i'_1,i'_1,i'_3,i'_3,\ldots,i'_{2p-1},i'_{2p-1} , J}. 
		\end{aligned}
	\end{equation}
	Using the Laplace-type expansion in lemma \ref{Laplace formulas for Pfaffian}, we have that
	\begin{equation}
		\begin{aligned}
			&\sum_{\mid I\mid =2k , \atop \mid J\mid =n-2k} sgn(I,I^c) \Pf(Y_{I})\det(Y_{I^c}^{J})
			=(-1)^{(k+1)+\binom{n}{2}}
			\Pf\begin{pmatrix}   Y_{[n]} & Y_{[n],J}\\
				Y_{J,[n]} & 0
			\end{pmatrix}.
		\end{aligned}
	\end{equation}
	Note that
	\[
	\det\begin{pmatrix}   Y_{[n]} & Y_{[n],J}\\
		Y_{J,[n]} & 0
	\end{pmatrix}
	=0,
	\]
	which implies \eqref{lemma-odd-operator2}.
\end{proof}

\bigskip
\centerline{\bf Acknowledgments}
\medskip
The work is supported in part by the National Natural Science Foundation of China grant nos.
12171303 and 12001218, the Humboldt Foundation, the Simons Foundation grant no. 523868,
and the Fundamental Research Funds for the Central Universities grant nos. CCNU22QN002 and CCNU24JC001.
							
\bigskip
\bigskip

\centerline{\bf Statement on Conflict of Interest}
\medskip
The authors have no conflict of interest to declare that are relevant to this article.

\bigskip
\bigskip

\bibliographystyle{amsalpha}

\end{document}